\documentclass{amsart}
\usepackage{amsmath,amsthm}
\usepackage{amsfonts,amssymb}
\usepackage{enumerate}
\usepackage[latin1]{inputenc}
\usepackage{accents,color}
\usepackage{graphicx}

\newcommand{\lvt}{\left|\kern-1.35pt\left|\kern-1.3pt\left|}
\newcommand{\rvt}{\right|\kern-1.3pt\right|\kern-1.35pt\right|}

\newtheorem{thm}{Theorem}[section]
\newtheorem{cor}[thm]{Corollary}
\newtheorem{lem}[thm]{Lemma}
\newtheorem{prop}[thm]{Proposition}

\theoremstyle{remark}

\newcommand{\lp}{\mathcal{LP}}

 \def\a{{\alpha}}
 
 \def\g{{\gamma}}

 \def\l{{\lambda}}

 \def\CL{{\mathcal L}}

 \def\CC{{\mathbb C}}
 
 \def\NN{{\mathbb N}}

 \def\RR{{\mathbb R}}

\newcommand{\wt}{\widetilde}
\newcommand{\wh}{\widehat}

\def\f{\frac}

\graphicspath{{./}}

\begin{document}

\title {Slater Determinants of Orthogonal Polynomials}

\author{Dimitar K. Dimitrov}
\address{Departamento de Matem\'atica Aplicada\\
 IBILCE, Universidade Estadual Paulista\\
 15054-000 Sa\~{o} Jos\'e do Rio Preto, SP, Brazil.}
 \email{dimitrov@ibilce.unesp.br}
\author{Yuan Xu}
\address{Department of Mathematics\\ University of Oregon\\
    Eugene, Oregon 97403-1222.}\email{yuan@uoregon.edu}

\date{\today}
\thanks{Research supported by the Brazilian foundations FAPESP under Grants 2009/13832--9 and 2014/08328--8, CNPq under Grant 307183/2013--0 and by the NSF under Grant DMS-1106113}

\begin{abstract}
The symmetrized Slater determinants of orthogonal polynomials with respect to a non-negative
Borel measure are shown to be represented by constant multiple of Hankel determinants of two
other families of polynomials, and they can also be written in terms of Selberg type integrals. 
Applications include positive determinants of polynomials of several variables and Jensen 
polynomials and its derivatives for entire functions. 
\end{abstract}

\maketitle

\section{Introduction}
\setcounter{equation}{0}

Let $f=\{f_i\}$ be a sequence of functions. For any given pair of nonnegative integers $n$ and $m$ and  
for fixed $t_1,\ldots, t_m \in \RR$, we consider the Slater determinant  
$$
S_{n,m} (f;t_1,\ldots,t_m) := 
\det\left[ \begin{matrix} f_n(t_1)  & \ldots & f_{n+m-1}(t_1)  \\
      \vdots & \ldots  &\vdots \\
     f_n(t_{m})  & \ldots & f_{n+m}(t_{m})   
    \end{matrix} \right] 
$$
and the symmetrized Salter determinant 
\begin{equation} \label{SSD}
W_{n,m}(f;t_1,\ldots,t_m) := \frac{S_{n,m} (f;t_1,\ldots,t_m)}{V(t_1,\ldots,t_m)}, 
\end{equation}
where $V(t_1,\ldots,t_m)$ is the Vandermond determinant of $t_1,\ldots, t_m$.  
When the variables coincide, the symmetrized Slater determinant becomes, up to a constant,
the Wronskian determinant
$$
W(f_n,\ldots, f_{n+m-1};x) :=  \det \left[ \begin{matrix} f_n(x) & f_{n+1}(x) & \ldots & f_{n+m-1}(x) \\
f_n'(x) & f_{n+1}'(x) & \ldots & f_{n+m-1}'(x) \\
 \cdots & \cdots & \cdots & \cdots \\
f_n^{(m-1)}(x) & f_{n+1}^{(m-1)}(x) & \ldots & f_{n+m-1}^{(m-1)}(x)
\end{matrix} \right].
$$

Slater determinants are wave functions of multi--particle fermion systems in Quantum Mechanics. For most of 
the models that are well understood, the wave functions are related to the classical orthogonal polynomials.
The main purpose of this paper is to study properties of Slater determinants for orthogonal polynomials.
 
Let $d \mu$ be a positive Borel measure on $\RR$ for which orthogonal polynomials $p=\{p_n\}$ exist. We shall
denote by $S_{n,m} (t_1,\ldots,t_m)$ and $W_{n,m}(t_1,\ldots,t_m)$ the Slater and symmetrized Slater determinant
$S_{n,m} (p;t_1,\ldots,t_m)$ and $W_{n,m} (p;t_1,\ldots,t_m)$ throughout this paper.  
Let $\mu_n$ be the $n$-th moment of $d\mu$, 
$$
   \mu_n := \int_\RR t^n d\mu(t), \qquad n=0,1,2,\ldots, 
$$
and let $M_n$ be the Hankel matrix of the moments defined by 
$$ 
M_n := \left[ \mu_{i+j} \right]_{i,j = 0}^{n} = \left[ \begin{matrix} 
  \mu_0 & \mu_1 & \cdots & \mu_n \\
  \mu_1 & \mu_2 & \cdots & \mu_{n+1} \\
  \cdots & \cdots &  \ddots & \cdots \\
  \mu_n & \mu_{n+1} & \cdots & \mu_{2n} 
   \end{matrix} \right].
$$
It is known \cite{Szego} that $\det M_n > 0$. The orthogonal polynomial $p_n$ with respect to
$d\mu$ can be defined by
\begin{align}  \label{eq:detPn}
p_n (x) & =  \det \left[ \begin{array}{c|c} 
    M_{n-1} & \begin{matrix} \mu_n  \\ \mu_{n+1} \\ \vdots \\ \mu_{2n-1} \end{matrix}\\ 
    \hline 1, x, \hdots x^{n-1} & x^n
      \end{array} \right], \qquad n =0,1,2,\ldots,    
\end{align}
which has the leading coefficient $\det M_{n-1}$. 

Associated with the measure $d\mu$, we define two additional sequences of polynomials,   
 \begin{align}
 q_n(x) & := \sum_{k=0}^n \mu_k \binom{n}{k} (-x)^{n-k},  \qquad n =0,1,2,\ldots.   \label{eq:def-qn} \\
 r_{m,n}(x) & := \sum_{k=0}^m \mu_{n+k} \binom{m}{k} (-x)^{m-k},  \qquad n, m = 0,1,2,\ldots.   \label{eq:def-rn} 
\end{align}
These polynomials may be viewed as shifted moments of $d\mu$ and moments of $(x-\cdot)^m d\mu$, respectively, 
because  
\begin{equation}\label{eq:q-integral}
  q_n (x) = \int_\RR (t-x)^n d\mu(t), \quad \hbox{and} \quad r_{m,n}(x) = \int_{\RR} t^n (t-x)^m d\mu(t).
\end{equation}
We define their extension to several variables by 
\begin{align} 
  q_n (t_1,\ldots,t_m; x) & := \int_\RR (t-x)^n   (t-t_1)\cdots(t-t_m)  d\mu(t), \label{eq:q-int}\\
  r_n(t_1,\ldots,t_m) & := \int_\RR t^n  (t-t_1)\cdots(t-t_m) d\mu(t). \label{eq:r-int}
\end{align} 

One of the main results in this paper shows that the symmetrized Salter determinant $W_{n,m}(t_1,\ldots,t_m)$ of orthogonal 
polynomials can be represented by the Hankel determinant of either $q_n(t_1,\ldots,t_{m-1};t_m)$ or  
$r_n(t_1,\ldots,t_m)$, and it can also be represent as a Selberg type integral. The results may be considered as 
generalizations of results obtained in a long paper by Karlin and Szeg\H{o}  \cite{KS} and in a more recent paper of 
Leclerc \cite{Lec}. The latter is our starting point. Indeed, when all variables coincide, our Theorem \ref{thm:main-pqr} below 
becomes
$$
 W(p_n, \ldots,p_{n+m-1};x) = C_{n,m}   \det \left[ q_{m+i+j}(x) \right]_{i,j =0}^{n-1} 
 =  C_{n,m}  \det \left[ r_{m, i+j} (x) \right]_{i,j =0}^{n-1}, 
$$
where $C_{n,m}$ is  constant. The first identity is exactly Leclec's result, whereas the second one appears to be new. Another 
consequence of our results shows that the function 
\begin{align*}
 F(x_1,\ldots,x_r): = \det\left[ \begin{matrix} p_n(x_1)& p_{n+1}(x_1) & \ldots & p_{n+2r-1}(x_1)  \\
   p_n'(x_1)& p_{n+1}'(x_1) & \ldots & p_{n+2r-1}'(x_1)  \\
    \vdots & \ldots & \ldots &\vdots \\
p_n(x_r)& p_{n+1}(x_r) & \ldots & p_{n+2r-1}(x_r)  \\
   p_n'(x_r)& p_{n+1}'(x_r) & \ldots & p_{n+2r-1}'(x_r) 
   \end{matrix}   \right]
\end{align*}
is nonnegative for all $n,r \in \NN$ and $(x_1,\ldots,x_r) \in \RR^r$. For $r=1$ this is a consequence of the Christoffel-Darboux 
formula. For $r>1$ this appears to be new and it is a special case of an even more general result.   

As an application, we show that the polynomials $q_n$ and $r_{m,n}$ are closely related to the Jensen 
polynomials of entire functions in the Laguerre-P\'olya class, and use our results to deduce new properties
for the Jensen polynomials. We will also discuss an interplay between our principal results and the orthogonal
polynomials that arise in the study of Toda lattices. 

The paper is organized as follows. The main result for the Slater determinants of orthogonal polynomials  is 
stated in the next section. The proof and further discussions are given in Section 3. The connection
with Jensen polynomials and Toda lattices are discussed in Section 4. Examples on classical orthogonal 
polynomials are given in Section 5. 

\section{Main Results on Slater determinants}
\setcounter{equation}{0}

To emphasize the dependence on the measure $d\mu$, we sometimes write $M_n(d\mu)=M_n$, $p_n(d\mu;x) = p_n(x)$ etc. 

\begin{thm} \label{thm:main-pqr}
For every $n, m \in \NN$, the symmetric Slater determinant $W_{n,m}(t_1,\ldots,t_m)$ obeys the identities   
\begin{align} 
W_{n,m}(t_1,\ldots,t_m)  & =    B_{n,m}  \det \left[ q_{i+j+1}(t_1,\ldots,t_{m-1}; t_m) \right]_{i,j =0}^{n-1}  \label{eq:Lec=q} \\
  &  =    B_{n,m}  \det \left[ r_{i+j}(t_1,\ldots,t_{m-1}, t_m) \right]_{i,j =0}^{n-1}, \label{eq:Lec=r}
\end{align}
where 
$$
   B_{n,m} = (-1)^{nm} \prod_{k=1}^{m-1} \det M_{k+n-1}.
$$ 
\end{thm} 

Let $\sigma_k(t_1,\ldots,t_m)$ denote the elementary symmetric functions of $t_1,\ldots,t_m$. By the definition of 
$r_n(t_1,\ldots,t_m)$, 
$$
   r_n (t_1,\ldots,t_m) = \sum_{k=0}^m (-1)^k \, \sigma_k(t_1,\ldots,t_m)\, \mu_{n+m-k}(d\mu).
$$
It follows form Theorem \ref{thm:main-pqr} that the symmetric Slater determinant can be written in a concise form in terms 
of the moments of $d\mu$. Indeed,  with $t=(t_1,\ldots,t_m)$,
$$
W_{1,m}(t)  = B_{1,m} r_0(t) = (-1)^m \prod_{k=1}^{m-1} \det M_k \sum_{k=0}^m (-1)^k \sigma_k(t) \mu_{m-k}(w),  
$$ 
which  appears in \cite{D}, and the next case is
$$
 W_{2,m}(t) = B_{2,m}  \left[r_2(t) r_0(t) - r_1(t)^2\right], \quad B_{2,m} = \prod_{k=1}^{m-1} \det M_{k+1}.
$$

We regard $q_n(x)$ defined in \eqref{eq:def-qn} as the case $m=0$ of $q_n(t_1, \ldots,t_m;x)$ and, evidently, 
$r_{m,n}(x)$ defined in \eqref{eq:def-rn} is $r_n(t_1, \ldots, t_m)$ with $t_1=\ldots = t_m =x$. Setting 
$t_1= \ldots = t_m =x$ in Theorem \ref{thm:main-pqr} gives the following corollary. 

\begin{cor} \label{thm:Lec}
For all integer values of $m, n \in \NN$, 
\begin{align} 
 W(p_n, \ldots,p_{n+m-1};x) & = C_{n,m}   \det \left[ q_{m+i+j}(x) \right]_{i,j =0}^{n-1} \label{eq:Lec} \\
  &  =  C_{n,m}  \det \left[ r_{m, i+j} (x) \right]_{i,j =0}^{n-1}, \label{eq:q=r}
\end{align}
where $C_{n,m}$ is a constant given by 
$$
  C_{n,m} := (-1)^{nm} \prod_{k=1}^{m-1} k! \det M_{k+n-1}.
$$
\end{cor}

Identity \eqref{eq:Lec} was proved by Leclerc \cite{Lec}. Notice that it is easy to see that the determinant 
$W(p_n, \ldots,p_{n+m-1};x)$ is a polynomial of degree $m n$, but this is not obvious for the determinant
$\left[ q_{m+i+j}(x) \right]_{i,j =0}^{n-1}$. It is clear, however, that 
$\det \left[ r_{m,i+j} (x) \right]_{i,j =0}^{n-1}$ is a polynomial of degree $mn$, since each $r_{m,i+j}(x)$ is of 
degree $m$. 

For a further generalization, we make the following definitions.  
 
Let $m_1,\ldots, m_r \in \NN$ and $m: = m_1 + \ldots + m_r$. We define the $m\times m$ determinant
\begin{align} \label{eq:Dn-det}
   S_n^{m_1,\ldots,m_r} (t_1,\ldots,t_r):= \det\left[ \begin{matrix} p_n(t_1)& p_{n+1}(t_1) & \ldots & p_{n+m-1}(t_1)  \\
   p_n'(t_1)& p_{n+1}'(t_1) & \ldots & p_{n+m-1}'(t_1)  \\
   \ldots & \ldots & \ldots &\ldots \\
   p_n^{(m_1-1)}(t_1)& p_{n+1}^{(m_1-1)}(t_1) & \ldots & p_{n+m-1}^{(m_1-1)}(t_1)  \\
 \vdots & \vdots & \vdots &\vdots \\
p_n(t_r)& p_{n+1}(t_r) & \ldots & p_{n+m-1}(t_r)  \\
   p_n'(t_r)& p_{n+1}'(t_r) & \ldots & p_{n+m-1}'(t_r)  \\
   \ldots & \ldots & \ldots &\ldots \\
   p_n^{(m_r-1)}(t_r)& p_{n+1}^{(m_r-1)}(t_r) & \ldots & p_{n+m-1}^{(m_r-1)}(t_r)  \end{matrix}   \right],
\end{align}
its symmetrized verson 
$$
W_n^{m_1,\ldots,m_r} (t_1,\ldots,t_r) :=  \frac{ S_n^{m_1,\ldots,m_r} (t_1,\ldots,t_r) }    { \prod_{1 \le i < j \le r}(t_j-t_i)^{m_i m_j} },
$$
as well as the corresponding polynomials
\begin{align}
 q_n^{m_1,\ldots,m_r}(t_1,\ldots,t_r; x): & = q_n\big(\overbrace{t_1,\ldots,t_1}^{m_1}, \ldots,
      \overbrace{t_r,\ldots,t_r}^{m_r}; x\big), \label{def:q-repeat}\\
   r_n^{m_1,\ldots,m_r}(t_1,\ldots,t_r): & = r_n\big(\overbrace{t_1,\ldots,t_1}^{m_1}, \ldots,  \overbrace{t_r,\ldots,t_r}^{m_r}\big).
 \label{def:r-repeat}
\end{align}

\begin{thm}  \label{cor:Slater}
For every $n\in \NN$, $m_1,\ldots,m_r \in \NN$ with $m: = m_1+\ldots+m_r$, 
\begin{align} 
 W_n^{m_1,\ldots,m_r} (t_1,\ldots,t_r) 
&  =  C_n^{m_1,\ldots,m_r} 
    \det \left[ q_{i+j+m_r}^{m_1,\ldots,m_{r-1}}(t_1,\ldots,t_{r-1}; t_r) \right]_{i,j =0}^{n-1} \label{eq:main2}\\
 & =  C_n^{m_1,\ldots,m_r} 
    \det \left[ r_{i+j}^{m_1,\ldots,m_{r}}(t_1,\ldots, t_r) \right]_{i,j =0}^{n-1},  \label{eq:main2r}
\end{align}
where
$$
     C_n^{m_1,\ldots,m_r} := (-1)^{n m} \prod_{i=1}^r \prod_{j=1}^{m_i} j! \prod_{k=1}^{m-1} \det M_{k+n-1}.
$$
\end{thm}

The matrix in \eqref{eq:main2r} is the moment matrix $M_{n-1}(w_{m_1,\dots,m_r}(t_1,\ldots,t_r))$ for the 
measure $w_{m_1,\ldots,m_r}(t_1,\ldots,t_r;x)= (x-t_1)^{m_1} \cdots (x-t_r)^{m_r} d\mu(x)$. 
We also obtain an integral representation for the determinants in \eqref{eq:Dn-det}.

\begin{thm} \label{thm:Delta=Int}
For $m_1,\ldots, m_r \in \NN$ and $n \in \NN$, 
\begin{align*}
  S_n^{m_1,\ldots,m_r} (t_1,\ldots,t_r) =\, & C_n^{m_1,\ldots,m_r} \prod_{1\le i < j \le r} (t_j - t_i)^{m_i m_j} \\
  & \times     \frac{1}{n!} \int_{\RR^n} \prod_{i=1}^r \prod_{j=1}^n (s_j-t_i)^{m_i}
        \prod_{1 \le i,j \le n} (s_i-s_j)^2 \prod_{j=1}^n d\mu(s_j). 
\end{align*}
\end{thm}

The integral in the above is a special case of the Selberg integral when $d\mu = w(x)dx $ and $w$ is a classical weight 
function. We refer to \cite{FW} for a beautiful account of the Selberg integrals.

An immediate consequence of our main result is the following remarkable corollary. 

\begin{cor} \label{cor:positivity}
Let $n, m_1,\ldots, m_r$ be positive integers. Then 
$$
   S_n^{2m_1,\ldots, 2 m_r} (t_1,\ldots, t_r) \ge 0 \qquad \mathrm{for\ every}\ \  (t_1,\ldots, t_r) \in \RR^r.  
$$
Furthermore, equality holds only if $t_i = t_j$ for some $i \ne j$.  
\end{cor}
 
In the case of $r =1$, $S_n^{m}$ is the Wronskian $W(p_n,\ldots, p_{n+m-1})$ and it is nonnegative 
on the real line if $m$ is even, as shown by Karlin and Szeg\H{o} in \cite{KS}. Our explicit integral representation 
gives a direct proof of this classical result.

We end this section by mentioning another connection of the Slater determinant. For $\a = (\a_1,\ldots,\a_m) \in 
\NN_0^m$ with $0 \le \a_1 \le \ldots \le \a_m =n$, define 
$$
   P_\a^n(u_1,\ldots,u_m) = \frac{\det \left[ p_{\a_{m-i+1} +i-1}(t_j)\right]_{i,j=1}^m}{V(t_1,\ldots,t_m)},
$$
where $u_i = \sigma_{m-i+1}(t_1,\dots,t_m)$, the elementary symmetric function of $t_1,\ldots,t_m$,
then $P_\a^n$ is a polynomial of degree $n$ in $(u_1,\dots,u_m)$ and, moreover, the set $\{P_\a^n:
0 \le \a_1\le \ldots \le \a_m =n\}$ is a complete set of orthogonal polynomials in $m$ variables 
(\cite[Section 5.4.1]{DX}). The Slater determinant corresponds to the case of $\a = (n,\ldots,n)$. 

\section{Proofs of the main results and further results} 
\setcounter{equation}{0}

We divide this section into subsections. The first subsection contains several lemmas on the polynomials
$p_n$, $q_n$ and $r_n$. Our main results on the determinants  are proved in the second subsection. The last 
section contains other related results on Slater determinants. 

\subsection{Lemmas}

We start with a fundamental tool in our work, an well-known identity that can be found in \cite[p. 62]{PS}.

\begin{lem} 
Let $f_i, g_j$ be functions such that $f_i g_j \in L^1 (\RR)$ for $1 \le i, j  \le n$. Then
\begin{equation} \label{eq:polya-szego}
  \det \left[ \int_{\RR} f_i(t) g_j(t) d\mu(t) \right]_{i,j=1}^n = \int_{\RR^n} \det \left[ f_j(t_i) \right]_{i,j=1}^n 
  \det \left[ g_j(t_i) \right]_{i,j=1}^n \prod_{i=1}^n d\mu(t_i).
\end{equation}
\end{lem}

Recall that the Vandermond determinant is given explicitly by
$$
   V(t_1,\ldots,t_m) = \det \left [ \begin{matrix} 1 & 1 & \ldots & 1 \\ t_1 & t_2 & \ldots & t_m \\
     \vdots & \ldots & \ldots & \vdots \\ t_1^{m-1} & t_2^{m-1} & \ldots & t_m^{m-1}
     \end{matrix}  \right] = \prod_{1 \le i < j \le m}(t_j-t_i). 
$$
One immediate consequence of the identity \eqref{eq:polya-szego} is an integral expression of the orthogonal 
polynomial $p_n(d\mu)$. 

\begin{lem} \label{lem:pn-detq}
For $n =0,1,\ldots$,
$$
  p_n(d\mu;x) = (-1)^n \det \left[ q_{i+j+1}(d\mu; x) \right]_{i,j =0}^{n-1}.
$$
\end{lem}

\begin{proof}
We use \eqref{eq:polya-szego} with $f_j(t) = (x- t)^j$ and $g_j(t) = (x-t)^{j+1}$ to obtain, by \eqref{eq:q-integral}, that
\begin{align} \label{eq:pn=det}
    \det \left[ q_{i+j+1}(x) \right]_{i,j =0}^{n-1} & = \frac{1}{n!}
       \int_{\RR^n} \det \left[ (t_i - x)^{j-1}\right]_{i,j=1}^{n} \det \left[ (t_i - x)^{j} \right]_{i,j=1}^{n}   
        \prod_{k=1}^n d\mu(t_k) \notag \\
   & = \frac{1}{n!}  \int_{\RR^n}  (t_1 - x) \cdots (t_n-x) \prod_{1 \le i<j \le n} (t_i - t_j)^2   
        \prod_{k=1}^n d\mu(t_k)  \notag \\
   & = (-1)^n p_n(x),
\end{align}
where the last identity follows from a well-known expression for orthogonal polynomial; see, for example,
\cite[p. 27]{Szego}.
 \end{proof}

This integral representation is the special case $m=1$ of \eqref{eq:Lec}, already established in \cite{Lec}. We
include the proof since it illustrates the strength of \eqref{eq:polya-szego}, which will be used several times in this
section. 

\begin{lem} \label{lem:qn}
The polynomial $q_n^{m_1,\ldots,m_r}(t_1,\ldots,t_r;\cdot)$ satisfies 
\begin{enumerate}[\rm (1)]
\item $ q_n(t_1,\ldots,t_m;x)= q_{n+1} (t_1,\ldots, t_{m-1};x) + (x- t_m)  q_n(t_1,\ldots, t_{m-1};x)$; 
\item $q_n(x,\ldots,x; x) = q_{n+m}(x)$.
\end{enumerate}
\end{lem}

\begin{proof}
The first item follows from \eqref{eq:q-int} by writing $t-t_m = t-x + x - t_m$. The second item follows directly from 
\eqref{eq:q-int}. 
\end{proof}

For given $t_1,\ldots, t_m$, let $d\nu_m$ be the measure defined by 
$$
      d\mu_m(x) = d\mu_m(t_1,\ldots,t_m; x): = (t-t_1)\cdots(t-t_m) d\mu(x).
$$

\begin{lem} \label{prop:q=r}
For $m , n = 1,2,\ldots,$ 
\begin{equation} \label{eq:q-moment}
  \det \left[ q_{i+j+1} (t_1, \ldots, t_{m-1}; t_m) \right]_{i,j=0}^{n-1} =  \det \left[ r_{i+j} (t_1,\ldots,t_m) \right]_{i,j=0}^{n-1}.
\end{equation}
\end{lem}

\begin{proof}
By the definition of $w_m$, we can write  
$$
  q_{i+j+1}(t_1, \ldots, t_{m-1}; t_m)  = \int_\RR (t-t_m)^{i+j} d\nu_m(t). 
$$
By \eqref{eq:polya-szego} with $f_j(t) = g_j(t) = (t-t_m)^j$, we see that 
\begin{align*}
 \det \left[ q_{i+j+1} (t_1, \ldots, t_{m-1}; t_m) \right]_{i,j=0}^{n-1}& =
   \int_{\RR^n} \left [V(s_1-t_m,\ldots, s_n-t_m) \right]^2  \prod_{i=1}^n d\nu_m(s_i)  \\
  & =  \int_{\RR^n} \left [V(s_1 ,\ldots, s_n) \right]^2  \prod_{i=1}^n d\nu_m(s_i),
\end{align*}
where we have used the closed form of the Vandermond determinant in the last step. Applying \eqref{eq:polya-szego} 
with $f_j(t)  = g_j(t) = t^j$, it follows that the last integral can be written as 
$$
\int_{\RR^n}  \left [V(s_1 ,\ldots, s_n) \right]^2 \prod_{i=1}^n d\nu_m(s_i) 
 =  \det  \left[ \int_{\RR} s^{i+j} d\nu_m(s) \right]_{i,j=0}^{n-1}. 
$$
Directly by the definition, the integral on the right-hand side is $r_{i+j}(t_1,\ldots,t_m)$, which proves \eqref{eq:q-moment}.
\end{proof}

We note that each $r_{i+j}$ is a symmetric function of $t_1,\ldots, t_m$, so that the right-hand side of the identity 
\eqref{eq:q-moment} is a symmetric function, which is, however, not obvious from the left-hand side of \eqref{eq:q-moment}
bacause $q_n(t_1,\ldots,t_{m-1};t_m)$ is not symmetric in $t_1,\ldots,t_m$. 

Lemma \ref{prop:q=r} shows that we have established the identity \eqref{eq:Lec=r} and, setting 
$t_1=\ldots = t_m$, the identity \eqref{eq:q=r}. Thus, we only need to prove our main theorems in terms of $q_n$, that is,
\eqref{eq:Lec=q}. 

For  $q_n^{m_1,\ldots,m_r}$, defined in \eqref{def:q-repeat}, we can write its Hankel determinant as an integral.

\begin{lem}
For $m_1,\ldots, m_r  \in \NN$ and $n \in \NN$,
\begin{align} \label{eq:detq=int}
 & \det \left[q_{i+j+1}^{m_1,\ldots,m_r}(t_1,\ldots,t_r; x) \right]_{i,j =0}^{n-1} \\
& \qquad = \frac{1}{n!} \int_{\RR^n}
    \prod_{i=1}^n (s_i-x)  \prod_{i=1}^r \prod_{j=1}^n (s_j-t_i)^{m_i}
        \prod_{1 \le i,j \le n} (s_i-s_j)^2 \prod_{j=1}^n d\mu(s_j). \notag
\end{align}
\end{lem}

\begin{proof}
Let $m = m_1+\ldots + m_r$. If $m_1 = \ldots = m_r =1$, then $r = m$ and, by \eqref{eq:pn=det} and \eqref{eq:pn-induct}, 
it follows that 
\begin{align*}
 & \det \left[ q_{i+j+1}(u_1,\ldots,u_m;x) \right]_{i,j =0}^{n-1} = (-1)^n p_n(d\nu_m;x) \\
 & \qquad\qquad =   
   \frac{1}{n!} \int_{\RR^n}  \prod_{i=1}^n (s_i-x)  \prod_{1 \le i,j \le n} (s_i-s_j)^2  
       \prod_{j=1}^n \prod_{i=1}^m (s_j-u_i) d\mu(s_j). 
\end{align*}
Setting $u_{1} = \ldots = u_{m_1} = t_1$, $u_{m_1+1} = \ldots = m_{m_1+m_2} = t_2$, $\ldots$, in the
above identity completes the proof. 
\end{proof}

Our last lemma in this subsection is well known. We give a proof since the same procedure will be used 
later. 

\begin{lem} \label{lem:Wronskin}
For $n, m \in \NN$, 
$$
  W_{n,m}(x,\ldots, x) =  \prod_{j=1}^{m-1} j! \ W(p_n, \ldots, p_{n+m-1}; x).
$$
\end{lem}

\begin{proof}
In the determinant $W_{n,m}(t_1, \ldots, t_m)$, we set $t_j = t_1+jh$ for $j=1,\ldots, m$ and rewrite 
the $j$-th row of the left hand side in terms of the forwarded difference $\triangle_h^{j-1} p_{n+i} (t_j)$. 
Since $\prod_{1 \le i < j \le m} (t_j- t_i) = \prod_{j=1}^{m-1} j! h^{m (m-1)/2}$, taking the limit 
$h \to 0$ completes the proof. 
\end{proof}

\subsection{Slater determinants}

We prove the following result from which Theorem \ref{thm:main-pqr} can be deduced.  

\begin{thm}  \label{thm:Slater}
For $n\in \NN$, $m_1,\ldots,m_r \in \NN$ and $m: = m_1+\ldots+m_r$, 
\begin{align} \label{eq:main}
 S_n^{m_1,\ldots,m_r,1} (t_1,\ldots,t_r,x) = & B_n^{m_1,\ldots,m_r} 
    \prod_{i=1}^r (x- t_i)^{m_i} \prod_{1 \le i < j \le r}(t_j-t_i)^{m_i m_j} \\
        & \times \det \left[ q_{i+j+1}^{m_1,\ldots,m_r}(t_1,\ldots,t_r;x) \right]_{i,j =0}^{n-1}, \notag
\end{align}
where
$$
     B_n^{m_1,\ldots,m_r} := (-1)^{n (m+1)} \prod_{i=1}^r \prod_{j=1}^{m_i} j! \prod_{k=1}^{m} \det M_{k+n-1}.
$$
\end{thm}

\begin{proof}
We first prove the case of $m_1 = \ldots = m_r =1$, for which $r =m$, by induction on $m$. That is, we prove
 \begin{align} \label{eq:Lec2}
    \det\left[ \begin{matrix} p_n(t_1)& p_{n+1}(t_1) & \ldots & p_{n+m}(t_1)  \\
      \vdots & \vdots & \vdots &\vdots \\
     p_n(t_m)& p_{n+1}(t_m) & \ldots & p_{n+m}(t_m)  \\
     p_n(x)& p_{n+1}(x) & \ldots & p_{n+m}(x) \end{matrix} \right]    
  = &   B_{n,m}  \prod_{i=1}^m (x- t_i) \prod_{1 \le i < j \le m}(t_j-t_i) \\
       & \times  \det \left[ q_{i+j+1}(t_1,\ldots,t_m; x) \right]_{i,j =0}^{n-1}, \notag
\end{align}
where $B_{n,m} = (-1)^{n(m+1)} \prod_{k=1}^{m} \det M_{k+n-1}$. 
For $m =0$, \eqref{eq:Lec2} is the identity in the lemma. We now assume that \eqref{eq:Lec2} holds 
for a fixed $m-1$ and prove that it holds with $m-1$ replaced by $m$. 

We can assume, without loss of generality, that $w(u)$ is supported on $[a,b]$ with a finite number $a$, 
since otherwise we can establish the identity for the truncated weight function $\chi_{a,b}(x) w(x)$, and then
take the limit $a \to -\infty$. Since the identity \eqref{eq:Lec2} is a polynomial in $t_1,\ldots, t_m$, the limit
exists. 

Since \eqref{eq:Lec2} is a polynomial identity, we only need to establish it for $t_1,\ldots,t_{m}$ less than $a$. Then 
$d\mu_m(t):=(t-t_1)\ldots (t-t_{m})d\mu(t)$ is a nonnegative weight function on $[a,b]$. It follows by the 
Lemma \ref{lem:pn-detq} that $p_n(d\nu_m)$ is given by 
\begin{equation} \label{eq:pn-induct}
   p_n(d\nu_m; x) = (-1)^n \det \left[ q_{i+j+1}(t_1,\ldots,t_{m}; x) \right]_{i,j =0}^{n-1}. 
\end{equation}
By \eqref{eq:q-moment} the leading coefficient $\g_n(d\nu_{m})$ of $p_n(d\nu_m;x)$ is given by 
\begin{equation}\label{eq:gn(wm)}
  \g_n(d\nu_m) = \det M_{n-1}(d\nu_m) = \det \left[ q_{i+j+1}(t_1,\ldots,t_{m-1};t_{m}) \right]_{i,j =0}^{n-1}.
\end{equation}
Moreover, by the Christoffel formula (\cite[p. 30]{Szego}), $p_n(d\mu_m)$ can also be given by 
\begin{equation} \label{eq:pn-induct2}
  p_n(d\nu_m;x) = \frac{A_{n,m}(t)}{\prod_{k=1}^{m} (x- t_j)} \det \left[ \begin{matrix} 
     p_n(t_1) & p_{n+1}(t_1) & \cdots & p_{n+m}(t_1) \\
     \vdots & \ldots & \ldots & \vdots \\ 
     p_n(t_{m}) & p_{n+1}(t_{m}) & \cdots & p_{n+m}(t_{m}) \\
     p_n(x) & p_{n+1}(x) & \cdots & p_{n+m}(x)  \end{matrix}\right],
\end{equation}
where $A_{n,m}(t) = A_{n,m}(t_1,\ldots, t_{m})$ is independent of $x$. In particular, the leading coefficient of 
$x^{n+m}$ in $\prod_{i=1}^m (x - t_i) p_n(d\nu_m;x)$, which is the same as $\g_n(d\nu_m)$, is given by 
$$
    \g_n(d\nu_m) = A_{n,m}(t)  \det M_{n+m-1} \det \left[ p_{n+j-1}(t_i) \right]_{i,j = 1}^{m}, 
$$
where we have used \eqref{eq:detPn}, from which it follows, by the induction hypothesis, that 
$$ 
  \g_n(d\nu_m) = A_{n,m}(t)  \det M_{n+m-1}B_{n,m} \prod_{1 \le i < j \le m}(t_j-t_i)
     \det \left[ q_{i+j+1}(t_1,\ldots,t_{m-1};t_m) \right]_{i,j =0}^{n-1}.
$$
Comparing the latter with \eqref{eq:gn(wm)} we obtain
$$
  \frac{1}{A_{n,m}(t)} = \det M_{n+m-1}B_{n,m} \prod_{1 \le i < j \le m}(t_j-t_i). 
$$
Consequently, combining \eqref{eq:pn-induct} and \eqref{eq:pn-induct2} proves 
\eqref{eq:Lec2} with $m-1$ replaced $m$, where the constant $B_{n,m}$ satisfies the relation 
$B_{n,m+1} = (-1)^n \det M_{n+m-1}B_{n,m}$. This completes the induction and the proof of \eqref{eq:Lec2}. 

Now we apply the limit procedure in Lemma \ref{lem:Wronskin} on the identity  \eqref{eq:Lec2}. Setting 
$t_j = t_1+jh$ for $j=1,\ldots, m_1$ in the identity and using 
\begin{align*}
   \prod_{1\le i< j \le m} (t_i - t_j) =   \prod_{1 \le i < j \le m_1} (t_i - t_j) \prod_{i=1}^{m_1} \prod_{j=m_1+1}^m 
         (t_i - t_j) \prod_{m_1+1\le i < j \le m} (t_i - t_j), 
\end{align*}
we take the limit $h \to 0$ to conclude that 
\begin{align*}
   W_n^{m_1,1 \ldots, 1,1} & (t_1, t_{m_r+1}, \ldots, t_m, x) =  B^{m_1,1,\ldots,1} 
   \prod_{j=m_1+1}^m (t_1 - t_j)^{m_1} \prod_{m_1+1\le i < j \le m} (t_i - t_j)  \\
   & \times (x-t_1)^{m_1} \prod_{i=m_1+1}^m (x-t_i) \det \left[ q_{i+j+1}^{m_1,1,\ldots,1}(t_1, t_{m_1+1} \ldots,t_m;x) \right]_{i,j =0}^{n-1},
\end{align*}
where the constant is given by
$$
   B^{m_1,1,\ldots,1}  = B^{1,1,\ldots,1}  \prod_{k=1}^{m_1-1} k!. 
$$
Repeating the above process by setting $t_{m_1+1} = \ldots = t_{m_1+m_2} = t_2$, so that 
\begin{align*}
   \prod_{j=m_1+1}^m (t_1 - t_j)^{m_1} =   
        (t_1 - t_2)^{m_1m_2}  \prod_{j=m_1+m_2+1}^m  (t_i - t_j)^{m_i},
\end{align*}
it follows that \eqref{eq:main} holds for $S_n^{m_1, m_2, 1,\ldots, 1,1}$. Continuing this process
completes the proof of \eqref{eq:main}. 
\end{proof}

We note that Theorem \ref{thm:Slater} is more general than Theorem \ref{thm:main-pqr}. Indeed, 
if $m_1=\ldots = m_r =1$, then \eqref{eq:main} becomes \eqref{eq:Lec2}, which is 
\eqref{eq:Lec=q} after replacing $x$ by $t_{m+1}$ and then replacing $m$ by $m-1$. Together with 
Lemma \eqref{prop:q=r}, this completes the proof of Theorem \ref{thm:main-pqr}. 

\medskip
{\it Proof of Theorem \ref{cor:Slater}.}
Taking $m_r$-th derivative of \eqref{eq:main} with respect to $x$ and then setting $x = t_r$, the left-hand
side becomes $S_n^{m_1,\ldots,m_{r-1}, m_r +1}(t_1,\ldots,t_r)$, whereas the constant in the right-hand side 
becomes $m! B_n^{m_1,\ldots,m_r}$ and the main term becomes
$$
      \prod_{1 \le i < j \le r-1}(t_j-t_i)^{m_i m_j} \prod_{i=1}^{r-1}(t_r-t_i)^{m_i (m_r+1)} 
        \det \left[ q_{i+j+1}^{m_1,\ldots, m_r}(t_1,\ldots,t_r; t_r) \right]_{i,j =0}^{n-1}. 
$$ 
By the definition of $q_n$, it is easy to see that 
$$
   q_{i+j+1}^{m_1,\ldots, m_r}(t_1,\ldots,t_r; t_r) =  q_{i+j+m_r+1}^{m_1,\ldots,m_{r-1}}(t_1,\ldots,t_{r-1}; t_r) 
$$
Replacing $m_r$ by $m_r -1$ in the resulted identity proves \eqref{eq:main2}. Then \eqref{eq:main2r} follows
from \eqref{eq:q-moment}.
\qed
\medskip

When $r=1$ and $t_1 =x$, the identity \eqref{eq:main2} becomes \eqref{eq:Lec}.
 
\medskip\noindent
{\it Proof of Theorem \ref{thm:Delta=Int}.} 
Combining \eqref{eq:main} with \eqref{eq:detq=int}, we obtain that 
\begin{align*}
 & S_n^{m_1,\ldots,m_r,1} (t_1,\ldots,t_r,x) =   B_n^{m_1,\ldots,m_r} 
    \prod_{i=1}^r (x- t_i)^{m_i} \prod_{1 \le i < j \le r}(t_j-t_i)^{m_i m_j} \\
        & \qquad \times \frac{1}{n!} \int_{\RR^n}
    \prod_{i=1}^n (s_i-x)  \prod_{i=1}^r \prod_{j=1}^n (s_j-t_i)^{m_i}
        \prod_{1 \le i,j \le n} (s_i-s_j)^2 \prod_{j=1}^n d\mu(s_j).
\end{align*}
The integral representation of $S_n^{m_1,\ldots,m_r} (t_1,\ldots,t_r)$ is deduced from comparing the leading 
coefficient of $x^{n+m}$ in the above identity. 
\qed
\medskip

We note that Corollary \ref{cor:positivity} follows immediately from Theorem \ref{thm:Delta=Int} and it also follows
from \eqref{eq:main2r}, since the right hand side of \eqref{eq:main2r} is the determinant of the moment matrix 
$M_{n-1}(d\mu_{m_1,\ldots,m_r})$ of the weight function $d\mu_{m_1,\ldots,m_r}(t) = (t-t_1)^{m_1}\cdots (t-t_r)^{m_r} d\mu(t)$,
which is nonnegative if $m_1,\ldots,m_r$ are even integers. 

\subsection{Further results on determinants}

For positive integers $\ell_1,\ell_2,\ldots, \ell_n$, we define 
$$
   F[q_{\ell_1},\ldots, q_{\ell_n}](t_1,\ldots,t_m;x) := \det \left[ q_{\ell_i+j-1}(t_1,\ldots,t_m;x) \right]_{i,j=1}^n.
$$
With this notation, the identity \eqref{eq:Lec2} becomes
$$
  \det \left[ p_{n+j-1}(t_i) \right]_{i,j =1}^{m} = B_{n,m-1} \prod_{1 \le i < j \le m}(t_j-t_i) 
     F[q_1,\ldots,q_n](t_1,\ldots, t_{m-1};t_m).
$$
Furthermore, for $1 \le j \le n+1$, we define 
$$
  F[q_1,\ldots, \wh q_j, \ldots, q_{n+1}]: = F[q_1,\ldots, q_{j-1},q_{j+1}, \ldots, q_{n+1}]. 
$$
 
\begin{lem}\label{lem:F=sumF}
For $m, n \in \NN$, 
$$
  F[q_m, q_{m+1}, \ldots,q_{m+n-1}](t;x) = \sum_{k=0}^{n} (x-t)^k F[q_m, \dots, \wh q_{m+k}, \ldots, q_{m+n}](x).
$$
\end{lem}

\begin{proof}
By Lemma \ref{lem:qn}, $q_k(t;x) = q_{k+1}(x) + (x-t) q_k(x)$. Using this relation and writing the determinant 
$F(q_1,\ldots, q_n)(t,x)$ as a sum of two determinants according to the fist row, we obtain
\begin{align*}
  & F[q_m,\ldots, q_{m+n-1}](t;x) \\
  & \qquad =   \left | \begin{matrix} q_{m+1}(x) & \cdots & q_{n+m}(x) \\
     q_{m+1}(t;x) & \cdots & q_{n+m}(t;x) \\
         \vdots & \cdots & \vdots \\
     q_{m+n}(t;x) & \cdots & q_{m+n}(t;x)     \end{matrix} \right|
 + (x-t) \left | \begin{matrix} q_m(x) & \cdots & q_{n+m-1}(x) \\
     q_{m+1}(t;x) & \cdots & q_{n+m}(t;x) \\
         \vdots & \cdots & \vdots \\
     q_{m+n}(t;x) & \cdots & q_{m+n}(t;x) 
   \end{matrix} \right|.
\end{align*}
Applying the relation $q_k(t;x) = q_{k+1}(x) + (x-t) q_k(x)$ multiple times, it is easy to see that the first
determinant simplifies to $F[q_{m+1},\ldots, q_{m+n}](x)$. For the second determinant, we repeat the
above procedure by splitting it into two determinants according to the second row, and simplify the
first one to $F[q_m, \wh q_{m+1}, q_{m+2}. \ldots, q_{m+n}](x)$. Continuing this process, it is easy to 
see that the last determinant is $F[q_{m},\dots,q_{m+n-1}](x)$. 
\end{proof}
 
\begin{prop}
For $m\le k \le n +m$, 
\begin{align}\label{eq:F-gap}
  \left | \begin{matrix} p_n(x) & \cdots & p_{n+m}(x) \\
         \vdots & \cdots & \vdots \\
        p_n^{(m-1)}(x) & \cdots & p_{n+m}^{(m-1)}(x) \\
       p_n^{(k)}(x) & \cdots & p_{n+m}^{(k)}(x) \end{matrix} \right|
     = & (-1)^{k-m} C_n^{1,m}\, k!\, F[q_{m},\ldots, \wh q_k, \ldots, q_{m+n}](x). 
\end{align}
\end{prop}

\begin{proof} 
Setting $r = 2$, $m_1=1$, $m_2 = m$, $t_1 = t$ and $t_2 =x$ in the identity \eqref{eq:main2}, Lemma \ref{lem:F=sumF}
yields
\begin{align*}
  \left | \begin{matrix} p_n(t) & \cdots & p_{n+m}(t) \\
       p_n(x) & \cdots & p_{n+m}(x) \\
         \vdots & \cdots & \vdots \\
        p_n^{(m-1)}(x) & \cdots & p_{n+m}^{(m-1)}(x) 
                 \end{matrix} \right|
     = \, &  C_{n}^{1,m}  (x -t)^m F[q_{m} \ldots, q_{n+m-1}](t;x) \\
     = \, & C_{n}^{1,m}  \sum_{j=0}^{n} (x-t)^{m+j} F[q_m, \dots, \wh q_{m+j}, \ldots, q_{m+n}](x).
\end{align*}
 Taking $k = m+j$ derivatives of the above identity with respect to $t$
and setting $t =x$, we obtain \eqref{eq:F-gap} after changing the first row to the last row. 
\end{proof}

If we want more gaps in the derivatives of $p_n$ in the determinant, we will need, by \eqref{eq:Lec2}, an extension 
of Lemma \ref{lem:F=sumF} to more than two variables. For example, by \eqref{eq:Lec2} and an obvious extension
of Lemma \ref{lem:F=sumF} to more variables,  
\begin{align*}
  & \left | \begin{matrix} 
  p_n(x) & p_{n+1}(x) & p_{n+2}(x) \\
  p_n(t_1) & p_{n+1}(t_1) & p_{n+2}(t_1) \\
  p_n(t_2) & p_{n+1}(t_2) & p_{n+2}(t_2)
     \end{matrix}\right |   
     = B_{n,2}  (t_2- t_1)(t_2-x)(t_1-x) F[q_1,\ldots,q_n](t_1,t_2;x) \\
   &  \qquad \qquad  = B_{n,2}  (x- t_1)(x-t_2)(t_2-t_1) \sum_{k=0}^n (x-t_1)^k F[q_1,\ldots, \wh q_{k+1}, \ldots q_{n+1}](t_2;x).
\end{align*}
Writing $t_2-t_1 = (x-t_1) - (x-t_2)$, then taking derivatives with respect to $t_1$ and setting $t_1=x$, it follows that 
\begin{align*}
  & \left | \begin{matrix} 
  p_n(x) & p_{n+1}(x) & p_{n+2}(x) \\
  p_n''(x) & p_{n+1}''(x) & p_{n+2}''(x) \\
  p_n(t_2) & p_{n+1}(t_2) & p_{n+2}(t_2)
     \end{matrix}\right |   
     = 2 B_{n,2} (x- t_2)  \\
   &  \qquad \qquad  \times \left( F[q_1,\ldots, q_n](t_2;x) -(x-t_2) F[q_2,\ldots,q_{n+1}](t_2;x) \right)    
\end{align*}
where we have used Lemma \ref{lem:F=sumF} in the second term.  
Consequently, expanding $F$ in terms of the power of $x-t_2$ by using Lemma \ref{lem:F=sumF}, we derive
the following: 

\begin{prop}
For $k \ge 2$, 
\begin{align*}
  & \left | \begin{matrix} 
  p_n(x) & p_{n+1}(x) & p_{n+2}(x) \\
  p_n''(x) & p_{n+1}''(x) & p_{n+2}''(x) \\
  p_n^{(k)}(x) & p_{n+1}^{(k)}(x) & p_{n+2}^{(k)}(t_2)
     \end{matrix}\right |   
     = 2 k! B_{n,2}   \\
   &  \qquad \qquad  \times \left( F[q_1,\ldots, \wh q_k, \ldots, q_n](x) - F[q_2,\ldots, \wh q_k, \ldots, q_{n+1}](x) \right).    
\end{align*}
\end{prop}
For $j > 2$, however, the above discussion leads to 
\begin{align*}
  & \left | \begin{matrix} 
  p_n(x) & p_{n+1}(x) & p_{n+2}(x) \\
  p_n^{(j)}(x) & p_{n+1}^{(j)}(x) & p_{n+2}^{(j)}(x) \\
  p_n(t_2) & p_{n+1}(t_2) & p_{n+2}(t_2)
     \end{matrix}\right |   
     = B_{n,2} j! (x- t_2)  \\
   &  \qquad \qquad  \times \left( F[q_1,\ldots, \wh q_{j-2}, \ldots q_{n+1}](t_2;x) -(x-t_2)
    F[q_1,\ldots, \wh q_{j-1}, \ldots q_{n+1}](t_2;x) \right).     
\end{align*}
In order to continue the above procedure, we have to expand the right-hand side
in powers of $x- t$,  for which we need a formula such as 
\begin{align*}
F[q_1,\ldots, \wh q_k, \ldots, q_{n+1}](t; x) = & \sum_{i=1}^k  (x-t)^{i-1}\sum_{j=k+1}^{n+2} (x-t)^{j-k-1} \\
    & \times F[q_1,\ldots, \wh q_i, \ldots, \wh q_j, \ldots, q_{n+2}](x). 
\end{align*}
Taking derivatives with respect to $t_2$ and setting $t_2 = x$, we can then write 
 $$
   \left | \begin{matrix} 
  p_n(x) & p_{n+1}(x) & p_{n+2}(x) \\
  p_n^{(j)}(x) & p_{n+1}^{(j)}(x) & p_{n+2}^{(j)}(x) \\
  p_n^{(k)}(x) & p_{n+1}^{(k)}(x) & p_{n+2}^{(k)}(x)
     \end{matrix}\right |   
 $$
as a sum of the determinants of the form $F[q_1,\ldots, \wh q_i, \ldots, \wh q_j, \ldots, q_{n+2}](x)$.

\section{Laguerre-P\'olya class of entire functions and Toda lattices} 
\setcounter{equation}{0}
\subsection{The Laguerre-P\'olya class}

The real entire function $\psi(x)$ is said to belong to the
Laguerre-P\'olya class $\lp$ if it can be represented as
$$
\psi(x) = c x^{m} e^{-a x^{2} + b x}
\prod_{k=1}^{\infty} (1+x/x_{k}) e^{-x/x_{k}},
$$
where $c, b$ and $x_{k}$ are real, $a\geq 0$, $m\in \NN_0$ and $\sum x_{k}^{-2} < \infty$. 
The functions in $\lp$, and only these, obey the property that they are local uniform limits, that is, uniform limits
on the compact subsets of $\mathbb{C}$, of polynomials with only real zeros. Such polynomials are usually called 
hyperbolic ones. The Laguerre-P\'olya class has been studied extensively since the Riemann hypothesis is 
equivalent to the fact that the Riemann $\xi$-function, the one that Titchmarsh denotes by $\Xi$, belongs to $\lp$. 
We refer to \cite{CraCso, CsoVar, D} and the references therein.

Laguerre gave a necessary condition for a function $\psi$ to be in the  Laguerre-P\'olya class 
$P \in \lp$: if $\psi \in \lp$ then
\begin{equation}\label{L_i}
L(\psi;x) = [\psi^\prime(x)]^2 - \psi(x) \psi^{\prime\prime}(x) \geq 0 \qquad \forall  x\in \mathbb{R}.
\end{equation}
Jensen established a necessary and sufficient condition. If $\psi \in \lp$ and its Maclaurin expansion is 
\begin{equation}\label{Maclaurin}
\psi(x) = \sum_{k=0}^{\infty} \gamma_{k} \frac{x^{k}}{k!}
\end{equation} 
then its Jensen polynomials are defined by
$$
g_{n}(x) = g_{n}(\psi;x) := \sum_{j=0}^{n} \binom{n}{j} \gamma_{j}x^{j}, \qquad  n = 0, 1, \ldots.
$$

\begin{prop} 
A function $\psi$ with the Maclaurin expansion (\ref{Maclaurin}) belongs to $\lp$ if and only if all its Jensen  
polynomials $g_{n}(\psi;x)$, $n\in \mathbb{N}$, are hyperbolic. 
Moreover, the sequence  $\{ g_{n}(\psi;x/n) \}$ converges locally uniformly to
$\psi(x)$.
\end{prop} 
 
The generalized Jensen polynomials are defined by 
$$
g_{n,k}(x) = g_{n,k}(\psi;x) := \sum_{j=0}^{n} \binom{n}{j} \gamma_{k+j} x^{j}, \quad n,k = 0, 1, \ldots.
$$
It is evident that $g_{n,0}(x)  = g_{n}(x)$ and $g_{n,k}(\psi;x) = g_{n}(\psi^{(k)};x)$, 
which shows, in particular, that $g_{n,k}(\psi;x/n) \rightarrow \psi^{(k)}(x)$ locally uniformly. 
Furthermore, it is easy to verify that $g_{n+k}^{(k)}(\psi;x) = \frac{(n+k)!}{k!} g_{n,k}(x)$. 
Consequently, it follows from \eqref{L_i} that, if  $\psi \in \lp$ then 
\begin{equation} \label{Ln_i}
L_n(\psi;x):= (n+2) [g_{n+1}'(\psi;x)]^2 - (n+1) g_{n}(\psi;x)\, g_n''(\psi;x) \geq 0,  \qquad  \forall x\in \mathbb{R}.
\end{equation}
Let us call $L_n(\psi;x)$ the Laguerre determinant of Jensen polynomials. 

Craven and Csordas \cite{CraCso} (see also \cite{CsoVar}) gave another criterion in terms of the
Tur\'an determinant of Jensen polynomials:

\begin{prop} 
Let the Maclaurin coefficients of the real entire function $\psi$ be such that $\gamma_{k-1}\gamma_{k+1}<0$ whenever $\gamma_k=0$, $k=1,2,\ldots$.
Then $\psi \in \lp$ if and only if 
\begin{equation}
\label{Ti}
T_n(\psi ; x) := g_{n}^2(\psi;x) - g_{n-1}^2(\psi;x) g_{n+1}^2(\psi;x) > 0, \quad  \forall x \in \mathbb{R}\setminus \{0\}\ \ \mathrm{and} \,\, n \in \mathbb{N}.
\end{equation}
\end{prop}

Our main result on the determinant shows that if $\psi$ is a Laplace transform of a non-negative measure, then 
the Laguerre polynomial inequalities and the Tur\'an inequalities are equivalent. Let us consider the bilateral Laplace transform
$$
     \CL_\mu(z): = \int_\RR e^{-z t} d\mu(t), \qquad z \in \CC, 
$$
for a real nonnegative measure $d\mu$ and it formal Maclaurin expansion
$$ 
 \CL_\mu (z) =  \sum_{k=0}^\infty \frac{\mu_k}{k!} (-z)^k,\qquad \mu_k := \int_0^\infty t^k d\mu(t). 
$$
Then its Jensen polynomials $g_n$ and $g_{n,k}$ are given by, with $\gamma_k = (-1)^k \mu_k$, 
$$
  g_n(\CL_ \mu; z) = \sum_{j=0}^n \binom{n}{j} (-1)^j \mu_j z^j \quad  \hbox{and}\quad
  g_{n,k}(\CL _\mu; z) = \sum_{j=0}^n \binom{n}{j} (-1)^{j+k} \mu_{j+k} z^j.
$$
A direct verification shows that 
\begin{equation}\label{eq:gn-D}
  g_n^{(j)}(\CL_\mu; z) = \frac{n!}{(n-j)!} g_{n-j,j}(\CL_\mu; z), \qquad 0 \le j \le n.
\end{equation}

It turns out that $g_n$ is related to our $q_n$ and $g_{n,k}$ is related to our $r_{k,n}$. 

\begin{lem} 
For $0 \le k \le n$, 
\begin{equation} \label{eq:g=q=r}
  g_n(\CL \mu; x) = (-x)^n q_n\left(d\mu; \frac{1}{x}\right) \quad \hbox{and} \quad 
  g_{n,k}(\CL \mu; x) = (-1)^{n+k} x^n r_{n,k}\left(d\mu; \frac{1}{x}\right). 
\end{equation}
\end{lem}

\begin{proof}
These follow directly from the definitions of Jensen's polynomials. 
\end{proof}

\begin{thm} \label{thm:det-g=g}
For $m, n \in \NN$, 
\begin{equation} \label{eq:det-g}
 \det [g_{m+i+j}(x)]_{i,j = 0}^{n-1} = x^{n(n-1)} \det \left [ \frac{m!}{(m+i+j)!} g_{m+i+j}^{(i+j)}(x)\right]_{i,j = 0}^{n-1}.
\end{equation}
\end{thm}

\begin{proof}
By \eqref{eq:g=q=r}, it is easy to see that 
\begin{align*}
  \det [g_{m+i+j}(x)]_{i,j = 0}^{n-1}& = x^{nm+n(n-1)} \det[q_{m+i+j}(1/x)]_{i,j =0}^{n-1}, \\
  \det \left [ \frac{m!}{(m+i+j)!} g_{m+i+j}^{(i+j)}(x)\right]_{i,j = 0}^{n-1} & = 
     x^{nm} \det \left[r_{m,i+j}(1/x) \right]_{i,j = 0}^{n-1},
\end{align*}
so that \eqref{eq:det-g} follows from \eqref{eq:q=r}.
\end{proof}
In particular, when $n =2$, the identity \eqref{eq:det-g} becomes 
\begin{align*}
&   [g_{m+1}(x)]^2 -g_m(x) g_{m+2}(x)  \\ 
   & \qquad\quad = \frac{x^2}{(m+2)(m+1)^2} 
      \left((m+2)[g_{m+1}'(x)]^2 - (m+1) g_m(x) g_{m+2}''(x)  \right),
\end{align*}
and it was observed by Craven and Csordas \cite{CraCso}. Identity \eqref{eq:det-g} gives a direct relation between 
the Tur\'an determinants and the Laguerre determinants of any order.

\subsection{Toda lattices and orthogonal polynomials}
 
 The Toda lattice is a model for a nonlinear one-dimensional crystal that describes the motion 
 of a chain of $N$  particles with nearest neighbour interactions. The  Hamiltonian of the Toda lattice is  
 $$
 H(\mathbf{p}, \mathbf{q}) = \sum_{k=1}^{N} \left( \frac{p_k^2(t)}{2} + e^{-(q_{k+1}(t) - q_{k}(t))} \right),
 $$
where $p_k$ is the moment of the $k$-th particle and $q_k$ is its displacement from the equilibrium. With the change of variables 
of Flashka and Moser
$$
a_k = \frac{1}{2} e^{-(q_{k+1} - q_{k})/2},\ \ b_k = - \frac{1}{2} p_k,
$$ 
the equations of motion become 
$$
a_k^\prime(t) = a_k(t) ( b_k(t) - b_k(t) ) \quad \hbox{and}\quad a_k^\prime(t) = 2 (a_k^2(t) - a_{k-1}^2(t)). 
$$
Let $L=(l_{i,j})$ be the Jacobi matrix with diagonal entries $l_{k,k}=b_k$ and off-diagonal ones $l_{k,k+1}=l_{k+1,k}
=a_k$, and let $B=(b_{i,j})$ whose only non-zero elements are the off-diagonal entries 
$b_{k,k+1}=-b_{k+1,k} = a_k$. Then the Lax form of the equations of motion is 
$$
\frac{d}{d\, t} L = [B,L].
$$
The matrix $L$ is naturally associated with the sequence of orthonormal polynomials, with the time variable as a parameter, which satisfy the three term recurrence relation 
$$
a_n(t)\, p_{n+1} (x;t) = (x-b_n(t))\, p_{n} (x;t) -  a_{n-1}(t)\, p_{n-1} (x;t).
$$
These are in fact the characteristic polynomials of the principal minors of $L$ and they are orthogonal with respect to the  
measure $d\mu_t(x) = e^{tx} d\mu_0(x)$, where $d\mu_0$ corresponds to $t=0$. Once the direct problem with the initial 
data at $t=0$ is solved and the polynomials $p_n(x,0)$ are obtained, one needs to solve the inverse problem. A 
fundamental problem is to construct $p_n(x,t)$. In order to do so, as it is seen from (\ref{eq:detPn}), it suffices to 
calculate the moments $m_k(t) = \int t^k e^{tx} d\mu_0(x)$. These moments are obtained by successive differentiation 
of $m_0(t)$ because $d m_k(t) /dt = m_{k+1}(t)$. Therefore, the principal task in solving the inverse problem is to 
determine  
$$
    m_0(t) = \int e^{tx} d\mu_0(x).
$$  
One possible approach is to approximate this formal Laplace transform by the Wronskians of the orthogonal polynomials $p_n(x;0) = p_n(d\mu;x)$.  

If $d \mu$ is a non-negative measure supported on $[0, \infty)$, then $\CL \mu(x)$ is the Laplace transform of
$\mu$. Let $(\CL \mu)^{(k)}$ be the $k$-th derivative of $\CL \mu$. Then, for $x \ge 0$,  
$$
  (\CL \mu)^{(k)}(x) = \int_\RR (-t)^k  e^{- x t} d\mu(t) = (-1)^k \mu^{(x)}_k, 
$$ 
where $\mu_n^{(x)}$ is the $k$-th moment of the measure $d \mu^{(x)} : = e^{-x t} d\mu(t)$. 
By \eqref{eq:gn-D} and $g_n((\CL \mu)^{(k)}; x)=g_{n,k} (\CL \mu;x)$, we can rewrite \eqref{eq:q=r} as 
\begin{equation} \label{eq:det-g-Phi} 
 \det [g_{m+i+j}(\CL \mu;x)]_{i,j = 0}^{n-1} = x^{n(n-1)} \det \left [  g_m( (\CL \mu)^{(i+j)}; x)\right]_{i,j = 0}^{n-1}.
\end{equation}
Furthermore, by \eqref{eq:g=q=r} and \eqref{eq:Lec}, we conclude that 
$$
   C_{n,m} \det \left [  g_m( (\CL \mu)^{(i+j)}; x)\right]_{i,j = 0}^{n-1}
      = x^{n m}  W(p_n, \ldots,p_{n+m-1};1/x).  
$$
In particular, when $n=1$, we obtain the following corollary: 

\begin{cor} \label{thm:gL=gL}
For $m \in \NN$, $x \in \RR$, 
\begin{equation*} 
   g_m(\CL \mu;x) = (-x)^{n m} \frac{W(p_1, \ldots,p_m;1/x)} {\prod_{k=1}^{m-1} k! \det M_k(d\mu)} \to \CL \mu(x), 
   \quad  m \to \infty.
\end{equation*}
\end{cor}

As another corollary of these relations, we deduce the following result: 

\begin{cor}
Let $\mu$ be a nonnegative Borel measure and assume that its Laplace transform $\CL \mu$ is real analytic
on $[0, \infty)$. Then for $n = 1,2, \ldots,$
$$
      \det \left [ (\CL \mu)^{(i+j)}(x)\right]_{i,j = 0}^{n-1} \ge 0, \quad x \in (0,\infty). 
$$
\end{cor}

\begin{proof}
From \eqref{eq:g=q=r} and Corollary \ref{cor:positivity}, the right hand side of \eqref{eq:det-g-Phi} 
is nonnegative if $m$ is an even positive integer. By its definition, it is easy to see that 
$$
  g_m(\CL \mu; x) = \int_0^\infty (1 - t x)^m d\mu(t). 
$$
It follows from dominant convergence theorem that 
$$
    \lim_{m \to \infty} g_{2m} \left((\CL \mu)^{(i+j)}; \frac{x}{2m} \right) = \lim_{m \to \infty} 
    \int_0^\infty \left(1 - \frac{t x}{2 m}\right)^{2m} t^{i+j} d\mu(t) = (\CL \mu)^{(i+j)}(x). 
$$
For fixed $n$, the above limit carries over to the determinant in the right hand side of \eqref{eq:det-g-Phi}. This
completes the proof. 
\end{proof}

In fact, since the determinant in the corollary is that of the moment matrix for $d \mu^{(x)}$, it is positive. 
 
\section{Examples}
\setcounter{equation}{0}

In this section we illustrate our main results in the case of classical orthogonal polynomials. We recall that
if $\wt p_n(d\mu)$ denotes the orthonormal polynomial of degree $n$ with respect to the measure $d\mu$, then
\begin{equation}\label{eq:onp}
  \wt p_n(d\mu;x) = \sqrt{\frac{\det M_{n-1}(d\mu)}{\det M_n(d\mu)} } x^n + \ldots, 
\end{equation}
which can be used to determine the determinant of $M_n(w)$ for the classical orthogonal polynomials.

\subsection{Hermite Polynomials}

For the weight function $w(x) = e^{-x^2}dx /\sqrt{\pi}$, which is normalized so that $\mu_0 =1$, its moments
$\mu_k$ are given by 
$$
   \mu_{2 k} = \frac{(2k)!}{2^{2k} k!}  \quad\hbox{and}\quad \mu_{2k+1} =0. 
$$
The corresponding orthogonal polynomials are the Hermite polynomials $H_n(x)$, 
$$
  H_n(x) = \sum_{0 \le k \le n/2} (-1)^k \frac{n!}{k!(n-2k)!} (2x)^{n-2k}, 
$$
normalized by $H_n(x) = 2^n x^n+ \ldots$. The orthonormal Hermite polynomial is $\wt H_n(x) = H_n(x) / \sqrt{2^n n!}$. Hence, it follows from
\eqref{eq:onp} that the determinant of the moment matrix $M_n$ for $w(x)$ satisfies
$$
   \det M_n = \frac{n!}{2^n} \det M_{n-1} = \ldots = \prod_{k=1}^n \frac{k!}{2^k}. 
$$
Since the orthogonal polynomial $p_n$ which appears in \eqref{eq:Lec} is given by 
$$
  p_n(x) = \det M_{n-1} \frac{H_n(x)}{2^n}, 
$$
it follows that in this case
$$
   \det \left[ p_{n+j-1}(t_i) \right]_{i,j =1}^{m} = \prod_{j=1}^m \frac{\det M_{n+j-2}}{2^{n+j-1} }
         \det \left[ H_{n+j-1}(t_i) \right]_{i,j =1}^{m}.
$$
Let $w(t_1,\ldots, t_m; x) = (x-t_1)\ldots (x-t_m)\, e^{-x^2}$. According to \eqref{eq:Lec=r}, we obtain that 
\begin{equation}\label{eq:main-Hermite}
   \frac{ \det \left[ H_{n+j-1}(t_i)/2^{n+j-1} \right]_{i,j =1}^{m}}{V(t_1,\ldots,t_m)}
     = (-1)^{n m} \frac{\det M_{n-1}(w(t_1,\ldots,t_m))}{\det M_{n-1}}, 
\end{equation}
where $M_{n-1}(w(t_1,\ldots,t_m))$ is the moment matrix of $w(t_1,\ldots, t_m)$. 
Furthermore, \eqref{eq:q-int} shows that 
$$
 q_{n+1} (t_1,\ldots, t_{m-1};t_m) w(t) dt = \int_\RR (t - t_m)^n w(t_1,\ldots,t_m; t)dt
$$
is the shifted moment of $w(t_1,\ldots,t_m)$. According to \eqref{eq:Lec=q}, we can replace 
the term $\det M_{n-1}(w(t_1,\ldots,t_m))$ in
\eqref{eq:main-Hermite} by the determinant of these shifted moments of $w_m$. 

By the definition of \eqref{eq:def-qn}, it follows readily that 
$$
  q_n(x) =  \sum_{j=0}^n \mu_j \binom{n}{j} (-x)^{n-j} = i^n \frac{H_n(i x)}{2^n}.
$$
If $u_k = b^k v_k$, then it is easy to verify that  
$$
 \det \left[ u_{\ell_i+j-1}  \right]_{i,j=1}^n = b^{\sum_{i=1}^n (\ell_i + i-1)} 
 \det \left[ b_{\ell_i+j-1} \right]_{i,j=1}^n.
$$
Using this identity, it follows that \eqref{eq:Lec} becomes, for the Hermite polynomials, 
\begin{align*}
  \det \left[ H_{n+j-1}^{(k-1)}(x) \right]_{k,j =1}^{m} & =  \frac{(-1)^{mn} 2^{\frac{m(m-1)}{2}}\, i^{n(n+m-1)}}{ 2^{\frac{n(n-1)}{2}}\prod_{k=m}^{n-1} k! }
   \det \left[ H_{m+k+j}(i x) \right]_{k,j =0}^{n-1} \\
 &=  \frac{(-1)^{mn} 2^{\frac{(m+n)( m+n-1)}{2}} } {\prod_{k=m}^{n-1} k! }
   \det \left[ r_{m,k+j}(x) \right]_{k,j =0}^{n-1}, 
\end{align*}
where $r_{m, k+j}(x)$ is the $(k+j)$-th moment of the weight function $w_m(t)=(t-x)^m e^{-t^2}$, and the first equality 
has appeared in \cite[(33)]{Lec}. 

\subsection{Laguerre polynomials}
For $\a > -1$, the weight function $w_\a(t) = t^\a e^{-t} /\Gamma(\a+1)$ has the moments
$\mu_k(w_\a) = (\a+1)_k$, $k=0,1,\ldots$. The Laguerre polynomials are orthogonal with respect to $w_\a$ and
they are explicitly given by 
$$
   L_n^\a(x) = \frac{(\a+1)_n}{n!} \sum_{j=0}^n \frac{(-n)_j x^j}{(\a+1)_j j!} =   \g_n x^n + \ldots, 
       \qquad \g_n := \frac{(-1)^n}{n!}. 
$$
The leading coefficient of the orthonormal Laguerre polynomial of degree $n$ is given by 
$1/\sqrt{n! (\a+1)_n}$, so that, by \eqref{eq:onp},
$$
  M_n(w_\a) = \prod_{k=1}^n \frac{M_k(w_\a)}{M_{k-1}(w_\a)} = \prod_{k=1}^n k! (\a+1)_k. 
$$
Let $w_\a(t_1,\ldots, t_m; x) := (x-t_1)\ldots (x-t_m) x^\a e^{-x}$. According to \eqref{eq:Lec=r}, we obtain that 
\begin{equation*}
   \frac{ \det \left[ L^\a_{n+j-1}(t_i)/\g_{n+j-1} \right]_{i,j =1}^{m}}{V(t_1,\ldots,t_m)}
     = (-1)^{n m} \frac{\det M_{n-1}(w_\a(t_1,\ldots,t_m))}{\det M_{n-1}(w_\a)}, 
\end{equation*}
where $M_{n-1}(w_\a(t_1,\ldots,t_m))$ is the moment matrix of $w_\a(t_1,\ldots, t_m)$.

By the definition of $q_n = q_n(w_\a)$ in \eqref{eq:def-qn}, we obtain that 
\begin{align*}
  q_n(w_\a;x)&  =\sum_{j=0}^n \binom{n}{j} (\a+1)_{n-j} (-x)^j \\
      &= (-1)^n (-n-\a)_n \sum_{j=0}^n \frac{(-n)_j (-x)^j }{(-n-\a)_j j!}  
        = (-1)^n n! L_n^{(-n-\a-1)}(-x).
\end{align*}
It follows that \eqref{eq:Lec} becomes, for the Laguerre polynomials, 
\begin{align*}
  \det \left[ (L^\a_{n+j-1}(x))^{(i-1)} \right]_{i,j =1}^{m}   =\, &
     \frac{(-1)^{m(m-1)/2} \prod_{k=1}^{m-1}k! }{\prod_{j=1}^m (n+j-1)! \det M_{n-1}(w_\a)} \\
     & \times   \det \left[ (m+i+j)! L_{m+i+j}^{-m-i-j-\a-1}(- x) \right]_{i,j =0}^{n-1} \\
 =\, &  \frac{(-1)^{m(m-1)/2} \prod_{k=1}^{m-1}k! }{\prod_{j=1}^m (n+j-1)! \det M_{n-1}(w_\a)}
     \det \left[ r^\a_{m,i+j}(x) \right]_{i,j =0}^{n-1}, 
\end{align*}
where $r^\a_{m, i+j}(x)$ is the $(i+j)$-th moment of the measure $(t-x)^m e^{-t} dt$. The polynomial $q_n(w_\a;x)$ 
appeared to be a constant multiple of $L_n^{-n-2\a}(x)$ in \cite{Lec}, but $-n -2\a$ should be $-n -\a-1$. 

\subsection{Gegenbauer polynomials} 
For $\l > -1/2$, the weight function $w_\l(t) = c_\l (1-t^2)^{\l-1/2}$, where $c_\l =
 \Gamma(\l+1)/( \Gamma(\f12) \Gamma(\l+\f12))$ is chosen so that $\mu_0 =1$ and the moments are given by 
$$
  \mu_{2k} =  \frac{(\f12)_k}{(\l+1)_k} \quad \hbox{and}\quad \mu_{2k+1} =0, \quad k =0,1,\ldots.
$$
The Gegenbauer polynomials $C_n^\l$ are orthogonal with respect to $w_\l$, they satisfy
$$
   c_\l \int_{-1}^1 C_n^\l(t) C_m^\l(t) w_\l(t) dt = h_n^\l \delta_{n,m}, \quad h_n^\l := \frac{\l(2\l)_n}{(n+\l) n!}, 
$$
and those polynomials are given explicitly as 
\begin{equation} \label{eq:gegenC}
   C_n^\l(x) = \g_n^\l x^n {}_2F_1\left( \begin{matrix} -\f{n}2, \f{1-n}{2} \\ 1-n-\l \end{matrix}; \f1{x^2} \right), 
   \qquad \g_n^\l: =  \frac{(\l)_n 2^n}{n!}.
\end{equation}
The leading coefficient $\g_n^\l$ of $C_n^\l$ divided by $\sqrt{h_n^\l}$ is the leading coefficient of the orthonormal Gegenbauer 
polynomial of degree $n$. Hence, by \eqref{eq:onp}, 
$$
  M_n(w_\l) = \prod_{k=1}^n \frac{M_k(w_\l)}{M_{k-1}(w_\l)}
     = \prod_{k=1}^n \frac{\l (2\l)_k k!}{(k+\l) (\l)_k^2 2^{2k}}  = \frac{\l^n}{(\l+1)_n} \prod_{k=1}^n \frac{(2\l)_k k!}{(\l)_k^2 2^{2k}}.
$$
Let $w_\l(t_1,\ldots, t_m; x) := (x-t_1)\ldots (x-t_m) (1-x)^{\l-1/2}$. According to \eqref{eq:Lec=r}, we obtain that 
\begin{equation}\label{eq:main-Gegenbauer}
   \frac{ \det \left[ C^\l_{n+j-1}(t_i)/\g^\l_{n+j-1} \right]_{i,j =1}^{m}}{V(t_1,\ldots,t_m)}
     = (-1)^{n m} \frac{\det M_{n-1}(w_\l(t_1,\ldots,t_m))}{\det M_{n-1}(w_\l)}, 
\end{equation}
where $M_{n-1}(w_\l(t_1,\ldots,t_m))$ is the moment matrix of $w_\l(t_1,\ldots, t_m)$.

\begin{lem}
For $w_\l$ and $n=0,1,\ldots$, $q_n = q_n(w_\l)$ is given by 
\begin{equation} \label{eq:qn=Cn}
   q_n(x) =  \frac{n!}{2^n (\l+1)_n} C_n^{-n-\l}(x) =  
     (x^2-1)^{n/2} \frac{ C_n^{\l+1/2} \left(-x/\sqrt{x^2-1}\right)}{ C_n^{\l+1/2}(1)}.
\end{equation}
\end{lem}

\begin{proof}
Directly from the definition of \eqref{eq:def-qn}, it is easy to see that 
\begin{equation} \label{eq:qn-Gegen}
   q_n(x) = \sum_{j=0}^n \binom{n}{j} \mu_j (-x)^{n-j} = (-x)^n
       {}_2F_1\left(\begin{matrix} - \f{n}{2}, \f{1-n}{2} \\ \l+1 \end{matrix} ; \f1{x^2} \right).
\end{equation}
Writing $\l +1 = 1-n -(-n-\l)$ and using $(-\l-n)_n = (-1)^n (\l+1)_n$, the first expression for $q_n$ follows from
\eqref{eq:gegenC}. Applying the identity \cite[(2.3.14)]{AAR}, 
$$
   {}_2F_1\left(\begin{matrix} -n, b \\ c \end{matrix} ; x \right) = \frac{(c-b)_n}{(c)_n}
       {}_2F_1\left(\begin{matrix} -n, b \\ b+1-n-c \end{matrix} ; 1- x \right) 
$$
to the right hand side of \eqref{eq:qn-Gegen}, it is easy to see that we obtain
\begin{align*}
q_n(x) = \, & (-x)^n \frac{(\l+\f12)_n 2^n}{(2\l+1)_n}
   {}_2F_1\left(\begin{matrix} - \f{n}{2}, \f{1-n}{2} \\ \f12 - n - \l \end{matrix} ; 1- \f1{x^2}  \right) \\
    = \, & \frac{n!}{(2\l+1)_n} (x_2-1)^{n/2} C_n^{\l+1/2} \left(-x/\sqrt{x^2-1}\right) 
\end{align*}
by \eqref{eq:gegenC}, which is the second representation of $q_n$ since $C_n^{\l+1/2}(1) = (2\l+1)_n/n!$.  
 \end{proof}

The second expression of $q_n$ in \eqref{eq:qn=Cn} has already appeared in \cite{Lec}. It follows that 
\eqref{eq:Lec} becomes, for the Gegenbauer polynomials, 
\begin{align*}
  \det \left[ (C^\l_{n+j-1}(x))^{(i-1)} \right]_{i,j =1}^{m}   =\, & \frac{(-1)^{mn} \prod_{k=1}^{m-1}k! 
    \prod_{j=1}^m \g_{n+j-1}^\l }{\det M_{n-1}(w_\l)}
     (x^2-1)^{m(m+n-1)/2} \\
   & \times  \det \left[ \frac{C_{m+i+j}^{\l+1/2}(- x/\sqrt{x^2-1})}{C_{m+i+j}^{\l+1/2}(1)}\right]_{i,j =0}^{n-1} \\
 =\, & \frac{(-1)^{mn} \prod_{k=1}^{m-1}k! 
    \prod_{j=1}^m \g_{n+j-1}^\l }{\det M_{n-1}(w_\l)}  \det \left[ r^\l_{m,i+j}(x) \right]_{i,j =0}^{n-1}, 
\end{align*}
where $r^\l_{m, i+j}(x)$ is the $(i+j)$-th moment of the measure $(t-x)^m(1-t^2)^{\l-1/2}dt$, and the first 
equality has appeared in \cite{Lec}.


\begin{thebibliography}{99}

\bibitem{AAR}
        G. Andrew, R. Askey and R. Roy,
        \textit{Special Functions}, 
        Encyclopedia of Mathematics and its Applications \textbf{71}, 
        Cambridge Univ. Press, Cambridge, 1999.


\bibitem{CraCso}
T. Craven and G. Csordas, Jensen polynomials and the Tur\'an and Laguerre inequalities, Pacific J. Math. 
136 (1989), 241--260.

\bibitem{CsoVar} 
G. Csordas and R. Varga, Necessary and sufficient conditions and the Riemann hypothesis, Adv. Appl. Math. 
11 (1990), 328--357.

\bibitem{D} 
        D. K. Dimitrov, 
        Lee-Yang measures and wave functions. arXiv:1311.0596 

\bibitem{DX} 
        C. Dunkl and Y. Xu, 
        \textit{Orthogonal Polynomials of Several Variables}, 2nd edition,
        Encyclopedia of Mathematics and its Applications, vol. \textbf{155}, Cambridge Univ. Press, 
        Cambridge, 2014.
        
\bibitem{FW} 
        P. J. Forrester and S. O. Warnaar, 
        The importance of the Selberg integral, 
        \textit{Bull. Amer. Math. Soc.} \textbf{45} (2008), 489--534.

\bibitem{KS}
        S. Karlin and G. Szeg\H{o},
        On certain determinants whose elements are orthogonal polynomials,
        \textit{J. Anal. Math.} \textbf{8} (1960), 1--157.         
        
        
\bibitem{Lec}
        B. Leclerc, 
        On certain formulas of Karlin and Szeg\H{o},
        \textit{S\'eminaire Lotharingien Combin.} \textbf{41} (1998), Art. B41d.

\bibitem{PS}
        G. P\'olya and G. Szeg\H{o},
        \textit{Problems and Theorems in Analysis}, Vol. I, 
        Springer-Verlag, Berlin and New York; 1972. 
   
        
\bibitem{Szego}
        G. Szeg\H{o},
        \textit{Orthogonal Polynomials}, 4th edition,
        Amer. Math. Soc., Providence, RI. 1975. 

      
 
\end{thebibliography}
\end{document}